\documentclass[a4paper,12pt]{article}
\usepackage{amsmath,fancyhdr,amssymb,amsthm, geometry, enumerate, graphicx ,amsfonts,hyperref,mathrsfs}
\usepackage{amsmath,bm}
\title{Several stronger forms of transitivity in non-autonomous discrete dynamical systems}
\author{Hongbo Zeng$^a$, \  $^{b,\dag}$}
\title{Several stronger forms of transitivity in non-autonomous discrete dynamical systems}
\theoremstyle{definition}

\providecommand{\keywords}[1]{\textbf{Keywords :} #1}

\theoremstyle{plain}

\newtheorem{definition}{Definition}
\newtheorem{lemma}{Lemma}
\newtheorem{remark}{Remark}

\newtheorem{theorem}{Theorem}
\newtheorem{example}{Example}
\newtheorem{corollary}{Corollary}
\begin{document}
\date{}
\maketitle

\begin{abstract}
In this paper, we introduce and study the notions of $\Delta$-mixing, $\Delta$-transitivity, mildly mixing, strong multi-transitivity and multi-transitivity with respect to a vector in non-autonomous discrete dynamical systems (NDS).
Firstly, we prove that multi-transitivity (strong multi-transitivity, multi-transitivity with respect to a vector, $\Delta$-transitivity, respectively) of NDS is iteration invariants. Then, necessary and sufficient conditions are obtained under which an NDS is strongly multi-transitive ($\Delta$-mixing, $\Delta$-transitive, respectively). Besides, we present some counterexamples to justify that the results related to stronger forms of transitivity which are true for autonomous discrete dynamical systems (ADS) but fail in NDS, which show that there is a significant difference between the theory of ADS and the theory of NDS, and establish a sufficient condition under which the results still hold in NDS. Finally, we give a sufficient condition under which multi-transitivity, weakly mixing, weakly mixing of all orders, thick transitivity, $\Delta$-transitivity and strongly multi-transitive are
equivalent in NDS.

\end{abstract}


\bigskip\renewcommand{\thefootnote}{\fnsymbol{footnote}}
\footnotetext{\hspace*{-5mm}
\renewcommand{\arraystretch}{1}
\begin{tabular}{@{}r@{}p{11cm}@{}}
$^a$&School of Mathematics and Statistics, Changsha University of Science and Technology, Changsha, P.R. China.\\
 $^b$&

\end{tabular}}
\vspace{-2mm}
\section{Introduction}

Dynamical systems theory is an effective mathematical mechanism which describes the time dependence of a point in a geometric space and has remarkable connections with different areas of mathematics such as topology and number theory. It is used to deal with the complexity, instability, or chaos in the real world, such as in meteorology, ecology, celestial mechanics, and other natural sciences.  Topological transitivity (shortly, transitivity) has been an eternal topic in the study of topological dynamical systems, which is a crucial measure of system complexity. The concept of transitivity can be traced back to Birkhoff \cite{q12}. Since then, many studies have been devoted to this topic.  Many scholars have done much work in classifying transitive systems by their recurrence properties. The relationship among variations on the concept of transitivity can be found in \cite{q27}. And the notions of multi-transitivity,  $\Delta$-transitivity and $\Delta$-mixing for autonomous discrete dynamical systems (Abbrev. ADS) were introduced by Moothathu in 2010 \cite{wx2}. He proved  that $\Delta$-transitivity implies weakly mixing, but there exist some strongly mixing systems which are not $\Delta$-transitive. He also showed that for minimal ADS weakly mixing and multi-transitivity are equivalent. In 2012, Kwietniak \cite{q37} have constructed examples showing that in general there is no relation between multi-transitivity and weakly mixing in ADS. Later, Chen \cite{q38}  introduced multi-transitivity with
respect to a vector and proved that multi-transitive system is Li-Yorke chaotic.  In 2019, authors have characterized multi-transitivity for ADS using different forms of shadowing and uniform positive entropy \cite{q39}.

\par
As a natural extension of ADS, non-autonomous discrete dynamical systems (NDS) are an important part of topological dynamical systems. Compared with classical dynamical systems (ADS), NDS can describe various dynamical behaviors more flexibly and conveniently. Indeed, most of the natural phenomena, whose behavior is influenced by external forces, are time dependent external forces. However, many of the methods, concepts and results of ADS are not applicable. Therefore, there is a strong need to study and develop the theory of NDS. The techniques used in this context are, in general, different from those used for ADS and make this discipline of great interest. In such systems the trajectory of a point is given by successive application of different maps. These systems are related to the theory of difference equations, and in general, they provide a more adequate framework for the study of natural phenomena that appear in biology, physics, engineering, etc \cite{wx3}. Meanwhile, the dynamics of NDS has became
an active research area, obtaining results on topological entropy, sensitivity, mixing properties, chaos, and other properties \cite{wx4,wx6,q24}.

The notion of NDS was introduced by Kolyada \cite{a14} in 1996. Since then, the study of complexity of NDS has seen remarkable increasing interest of many researchers, see \cite{q16,q18,q14,q23,wx1,q25} and the references therein. Sharma \cite{q36} related the dynamical behavior of $(X, f_{ 1,\infty} )$ with the dynamical behavior
of the limiting system (and vice versa). Recently, Salman \cite{wx1} introduced and studied the notions of multi-transitivity and thick transitivity for NDS. The authors obtained a sufficient condition under
which in minimal NDS, multi-transitivity,
thick transitivity and weakly mixing of all orders are equivalent.  The interested reader in transitivity  for NDS might consult \cite{wx1,wx5,wx7,q29}.

\par
Motivated by the above results, this paper is devoted to further study the properties of stronger forms of transitivity in NDS. This paper is organized as follows. In Section 2, we will state some preliminaries, definitions and lemmas. The main conclusions will be given in Section 3.  At first, we prove that multi-transitivity of NDS is iteration invariants. Then, we present some equivalent conditions for strongly multi-transitivity of NDS.  Next, we prove that mildly mixing implies multi-transitivity, which answers the open problems 2 in \cite{wx1}. Besides, we give a equivalent condition for $\Delta$-mixing ($\Delta$-transitivity, respectively) of NDS and prove that $\Delta$-transitivity of NDS is preserved under iterations. Meanwhile, we explore the relationship of transitivity between ADS and NDS for periodic NDS and the relationship of transitivity between two semi-conjugate NDS. Moreover, we provide some counterexamples to justify that the results related to transitivity which are true for ADS but not true for NDS, and provide a sufficient condition under which the results still hold in NDS.
   At last, we give several sufficient conditions under which multi-transitivity, weakly mixing, weakly mixing of all orders, thick transitivity, $\Delta$-transitivity and strongly multi-transitive are
equivalent.

\section{Preparations and lemmas}
In this section, we mainly give some different concepts of transitivity for NDS (see,for example,\cite{q18,q14,wx1,q36}) and some lemmas required for remaining sections of the paper.

Assume that $\mathbb{N}=\{1,2,3,...\}$. Let $(X,d)$ be a compact metric
space and $f_n : X \rightarrow X$ be a sequence of continuous functions, where $n\in \mathbb{N}$. An NDS is a pair $(X, f_{ 1,\infty} )$, where $f_{ 1,\infty}= \{f_n\}_{n=1}^\infty$. For any $i,n\in \mathbb{N}$,
define the composition
$f_i^n:= f_{i+(n-1)} \circ \cdot\cdot\cdot \circ f_i$,
and usual $f_i^0=id_X$, where $id_X$ is the identity map on $X$. In particular, if $f_n = f$ for all $n \in \mathbb{N}$, then the pair $(X, f_{ 1,\infty} )$ is
just the ADS $(X, f)$. The orbit of
a point $x$ in $X$ is the set
$orb_{f_{ 1,\infty}}(x):=\{x,f_1^1(x),f_1^2(x),...,f_1^n(x),...\}$,
which can also be described by the difference equations $x_0 = x$ and $x_{n+1} = f_n (x_n )$, $n\in \mathbb{N}$. Given $k \in \mathbb{N}$, $f_{ 1,\infty}^{[k]}:=\{f_{k(n-1)+1}^k\}_{n=1}^\infty$ is called the $k^{th}$-iterate of NDS $(X, f_{ 1,\infty} )$. Denote $f_1^{-n}=(f_1^{n})^{-1}=f_1^{-1}\circ f_2^{-1}\circ\cdots\circ f_n^{-1}$, $B(x,\varepsilon)$ the open ball of radius $\varepsilon>0$ with center $x$ and $int(A)$ the interior of $A$. For $r\in\mathbb{N}$, write $\mathbb{N}^r=\mathbb{N}\times\mathbb{N}\times\cdots\mathbb{N}$ ($r$-copies) and $X^r=X\times X\times\cdots X$ ($r$-copies).
If $\mathcal{C}(X)$ is the collection of continuous self maps on $X$, then for any $f,g\in\mathcal{C}(X)$, $D(f,g)=sup_{x\in X}d(f(x),g(x))$ is a metric on $\mathcal{C}(X)$ known as the $Supremum$ $metric$. A  collection of sequences $\{f_n^{n+k}\}_{k\in\mathbb{N}}$ converges collectively to $\{f^k\}_{k\in\mathbb{N}}$ with respect to the metric $D$ if for any $\varepsilon>0$, there exists $r_0\in\mathbb{N}$ such that $D(f_r^k,f^k) < \varepsilon$ for all $r\geq r_0$ and for every $k\in\mathbb{N}$ \cite{q36}.
We always suppose that $(X, f_{ 1,\infty} )$  is an NDS and that all the maps are continuous from $X$ to $X$ in the following context.

\begin{definition}
An NDS $(X, f_{ 1,\infty} )$  is said to be (topologically) transitive, if for any two non-empty subsets $U,V\subseteq X$, there exists $n\in \mathbb{N}$ such that $f_1^n(U)\cap V\neq\emptyset$.

An NDS $(X, f_{ 1,\infty} )$  is said to be weakly mixing, if for any four non-empty subsets $U_1,U_2,V_1,V_2\subseteq X$, there exists $n\in \mathbb{N}$ such that $f_1^n(U_1)\cap V_1\neq\emptyset$ and $f_1^n(U_2)\cap V_2\neq\emptyset$.

An NDS $(X, f_{ 1,\infty} )$  is said to be weakly mixing of order $n(n\geq2)$, if for any collection of non-empty subsets $U_1,U_2,...,U_n,V_1,V_2,...,V_n\subseteq X$, there exists $m\in \mathbb{N}$ such that $f_1^m(U_i)\cap V_i\neq\emptyset$ for each $i=\{1,2,...,n\}$.

An NDS $(X, f_{ 1,\infty} )$  is said to be mixing, if for any two non-empty subsets $U,V\subseteq X$, there exists $N\in \mathbb{N}$ such that $f_1^n(U)\cap V\neq\emptyset$ for any $n\geq N$.

An NDS $(X, f_{ 1,\infty} )$  is said to be totally transitive, if for any $n\in \mathbb{N}$, $f_{ 1,\infty}^{[n]}$ is transitive.

An NDS $f_{ 1,\infty}$ is called multi-transitive if $f_{ 1,\infty}\times f_{ 1,\infty}^{[2]}\times\cdots\times f_{ 1,\infty}^{[m]}:X^m\rightarrow X^m$ is transitive for any $m\in \mathbb{N}$. Equivalently, if for any $m\in \mathbb{N}$ and for any collection of nonempty open subsets $U_1,U_2,...,U_m;V_1,V_2,...,V_m$ of $X$, there exists $l\in \mathbb{N}$ such that  $f_1^{il}(U_i)\cap V_i\neq\emptyset$ for each $i\in\{1,2,...,m\}$.

\end{definition}

\begin{definition}
 A point $x\in X$ is said to be transitive if the orbit of $x$ is dense in $X$. An NDS $(X, f_{ 1,\infty} )$  is said to be minimal if all the points of $X$ are transitive.
\end{definition}

\begin{definition}
A family $f_{ 1,\infty}$ is said to be commutative (abelian), if each of its member commutes with every other member of the family, that is, $f_i\circ f_j=f_j\circ f_i$ for any $i,j\in \mathbb{N}$.

An NDS $(X, f_{ 1,\infty} )$ is said to be $k$-periodic, if there exists $k\in \mathbb{N}$ such that $f_{n+k}(x)=f_{n}(x)$ for any $x\in X$ and any $n\in \mathbb{N}$.

 An NDS $(X, f_{ 1,\infty} )$  is said to be feeble open if $int(f_i(U))\neq\emptyset$ for any nonempty open set $U\subseteq X$ and any $i\in \mathbb{N}$.

\end{definition}

Next we generalize the concept of multi-transitivity with respect to a vector (strong multi-transitivity, mildly mixing, $\Delta$-transitivity, $\Delta$-mixing, respectively) from ADS to NDS.

\begin{definition}

Let $p\in\mathbb{N}$ and $\textbf{a}$ = $(a_1, a_2, \ldots, a_p)$ be a vector in  $\mathbb{N}^p$.

An NDS $(X, f_{ 1,\infty} )$  is said to be multi-transitive with respect to the vector $\textbf{a}$ if the system $(X^p, f_{1, \infty}^{[\textbf{a}]})$ is transitive, where $f_{1, \infty}^{[\textbf{a}]} = f_{1, \infty}^{[a_1]} \times f_{1, \infty}^{[a_2]} \times \cdots \times f_{1, \infty}^{[a_p]}$;

An NDS $(X, f_{ 1,\infty} )$  is said to be strongly multi-transitive if $(X, f_{ 1,\infty} )$  is multi-transitive with respect to any vector in $\mathbb{N}^n$ and any $n\in\mathbb{N}$.

An NDS $(X, f_{ 1,\infty} )$  is said to be mildly mixing, if for any transitive NDS $(Y, g_{ 1,\infty} )$, the product system $(X\times Y, f_{ 1,\infty}\times g_{ 1,\infty} )$ is transitive.

 An NDS $(X, f_{ 1,\infty} )$  is said to be $\Delta$-transitive, if for any $m\in \mathbb{N}$, there exists a dense $G_\delta$ subset $Y\subset X$ such that for every $x\in Y$, $\{(f_1^{n}(x),f_1^{2n}(x),...,f_1^{mn}(x))\mid n\in \mathbb{N}\}$ is dense in $X^m$.

An NDS $(X, f_{ 1,\infty} )$  is said to be $\Delta$-mixing, if for each $m\in \mathbb{N}$ and infinite subset $A\subset \mathbb{N}$, there exists a dense $G_\delta$ subset $Y\subset X$ such that for every $x\in Y$, $\{(f_1^{n}(x),f_1^{2n}(x),...,f_1^{mn}(x))\mid n\in A\}$ is dense in $X^m$.
\end{definition}

For the convenience of the following context, we denote $$N_{f_{ 1,\infty}}(U,V)=\{n\in \mathbb{N} \mid f_1^n(U)\cap V\neq\emptyset\}$$  for any nonempty open sets $U,V$ of $X$.

\begin{definition}
A set $A\subseteq \mathbb{N}$ is called syndetic if it has bounded gaps, i.e. there exists a positive integer $M$ such that $\{i,i+1,...,i+M\}\cap A\neq\emptyset$ for every $i\in \mathbb{N}$. A set $A\subseteq \mathbb{N}$ is called cofinite if there exists $N\in \mathbb{N}$ such that $A\supseteq[N,\infty)\cap \mathbb{N}$. A set $A\subseteq \mathbb{N}$ is called thick if it contains arbitrarily long runs of positive integers, that is, for any $p\in \mathbb{N}$, there exists some $n\in \mathbb{N}$ such that $\{n,n+1,...,n+p\}\subseteq A$.  A set $A\subseteq \mathbb{N}$ is called thickly syndetic if $\{m\in \mathbb{N}: m+j\in A, \text{for } 0\leq j\leq l\}$ is syndetic for any $l\in \mathbb{N}$.

An NDS $(X, f_{ 1,\infty} )$  is said to be syndetically (thickly, thickly syndetically, respectively) transitive, if for any two non-empty subsets $U,V\subseteq X$, the set $N_{f_{ 1,\infty}}(U,V)$ is syndetical (thick, thickly syndetic, respectively).
\end{definition}

\begin{definition}
Let $(X,f_{ 1,\infty})$ and $(Y,f_{ 1,\infty})$ be two NDS. We say that $f_{1,\infty}$ is topologically semi-conjugate to $g_{1,\infty}$, if there is a continuous and surjective function $h:X\rightarrow Y$ such that $g_n(h(x))=h(f_n(x))$ for any $n\in \mathbb{N}$ and any $x\in X$.
\end{definition}

\begin{definition}
Let $(X,f_{ 1,\infty})$ be an NDS and $n\in\mathbb{N}$. Assume that  $x_0=x,x_1,...,x_n=y\in X$ and
$\delta>0$. If $d(f_{i+1}(x_i),x_{i+1})<\delta$ for any $i\in\{0,1,...,n-1\}$, the sequence $x_0=x,x_1,...,x_n$ is called a $\delta$-chain of $f_{ 1,\infty}$ from $x$ to $y$ with length $n$. A map $f_{ 1,\infty}$ is called chain transitive if for any two points $x,y\in X$ and any $\delta>0$ there exists a $\delta$-chain from $x$ to $y$. A map $f_{ 1,\infty}$ is called chain mixing if for any two points $x,y\in X$ and any $\delta>0$, there is an $N\in\mathbb{N}$ such that for any $n\geq N$ there is a $\delta$-chain from $x$ to $y$ with length $n$.

\end{definition}

\begin{lemma}(\cite[Theorem 4]{wx9})\label{yinli1}
Let $(X,f_{1,\infty})$ be an NDS. If $f_{1,\infty}$ is mildly mixing, then $f_{1,\infty}$ is mixing. Converse is true if we consider the system without isolated points.
\end{lemma}

\begin{lemma}\label{yinli2}

Let $(X,f_{1,\infty})$ be a non-autonomous discrete system. If
$f_{1,\infty}$ is strongly multi-transitive, then the set $\{l\in\mathbb{N}:f_1^{la_j}(U_j)\cap V_j\neq\emptyset,\forall j=1,2,...,m\}$ for any $m\in\mathbb{N}$, any vector $\bm{a}=(a_1,a_2,...,a_m)\in\mathbb{N}^m$  and any collection of nonempty open subsets $U_1,U_2,...,U_m;V_1,V_2,...,V_m$ of $X$, is infinite.

\end{lemma}

\begin{proof}
It can be proved by similar arguments as given by authors in \cite[Lemma 3.1]{wx1}.
\end{proof}
\section{Main results}

\begin{theorem}
Let $(X,f_{1,\infty})$ be an NDS and $n\in\mathbb{N}$. Then
 $f_{1,\infty}$ is multi-transitive if and only if $f_{1,\infty}^{[n]}$ is multi-transitive.
\end{theorem}
\begin{proof}
Necessity. Let $m\in\mathbb{N}$ and $U_1,U_2,\ldots,U_m$, $V_1,V_2,\ldots,V_m$ be non-empty open subsets of $X$. For each $i\in\{1,2,...,n\}$, take
$$U_{i}^{'} = U_1, \ \ U_{n+i}^{'} = U_2,\ \ \ldots,\ \ U_{(m-1)n+i}^{'} = U_m,$$
$$V_{i}^{'} = V_1, \ \ V_{n+i}^{'} = V_2,\ \  \ldots, \ \ V_{(m-1)n+i}^{'} = V_m.$$
Since $f_{1,\infty}$ is multi-transitive, there exists an $l\in\mathbb{N}$ such that
$f_1^{jl}(U_j^{'})\cap V_j^{'}\neq\emptyset$ for each $j\in\{1,2,...,mn\}$. In consequence, $f_1^{jl}(U_j^{'})\cap V_j^{'}\neq\emptyset$ for each $j\in\{n,2n,...,mn\}$, that is, $f_1^{njl}(U_j)\cap V_j\neq\emptyset$ for each $j\in\{1,2,...,m\}$. Therefore, $f_{1,\infty}^{[n]}$ is multi-transitive.

Sufficiency.  Let $m\in\mathbb{N}$ and $U_1,U_2,\ldots,U_m$, $V_1,V_2,\ldots,V_m$ be non-empty open subsets of $X$. Since $f_{1,\infty}^{[n]}$  is multi-transitive, there exists an $l\in\mathbb{N}$ such that $f_1^{njl}(U_j)\cap V_j\neq\emptyset$ for each $j\in\{1,2,...,m\}$. That is, there exists a $k=nl\in\mathbb{N}$ such that $f_1^{kj}(U_j)\cap V_j\neq\emptyset$ for each $j\in\{1,2,...,m\}$. Therefore, $f_{1,\infty}$ is multi-transitive.

\end{proof}

\begin{theorem}\label{strong4}
Let $(X,f_{1,\infty})$ be an NDS, $n\in\mathbb{N}$ and $\bm{a}=(a_1,a_2,...,a_r)\in\mathbb{N}^r$. Then
 the following conditions are equivalent:

 (1) $(X,f_{1,\infty})$ is strongly multi-transitive;

 (2) $(X,f_{1,\infty}^{[n]})$ is strongly multi-transitive;

 (3) $(X^r,f_{1,\infty}^{[\bm{a}]})$ is strongly multi-transitive;

 (4) for each $k\in\mathbb{N}$, $(X^k,f_{1,\infty}\times f_{1,\infty}^{[2]}\times\cdots f_{1,\infty}^{[k]})$ is weakly mixing of all orders.
\end{theorem}
\begin{proof}
$(1) \Rightarrow (2)$. Let $l\in\mathbb{N}$ and $\bm{a}'=(a_1',a_2',...,a_l')\in\mathbb{N}^l$. Since $(X,f_{1,\infty})$ is strongly multi-transitive, $(X,f_{1,\infty})$ is multi-transitive with respect to the vector $(na_1',na_2',...,na_l')$, which implies that $(X^l,(f_{1,\infty}^{[n]})^{(\bm{a}')})$ is transitive. Then $(X,f_{1,\infty}^{[n]})$ is  multi-transitive with respect to the vector $\bm{a}'$, and therefore $(X,f_{1,\infty}^{[n]})$ is also strongly multi-transitive.

$(2) \Rightarrow (3)$. Let $l\in\mathbb{N}$ and $\bm{a}'=(a_1',a_2',...,a_l')\in\mathbb{N}^l$. Since $(X,f_{1,\infty}^{[n]})$ is strongly multi-transitive, $(X,f_{1,\infty}^{[n]})$ is multi-transitive with respect to the vector $$(a_1a_1',a_2a_1',...,a_ra_1',a_1a_2',a_2a_2',...,a_ra_2',...,a_1a_l',a_2a_l',...,a_ra_l'),$$
which implies that $(X^{rl},(f_{1,\infty}^{(\bm{a})})^{(\bm{a}')})$ is transitive. Thus, $(X^r,f_{1,\infty}^{\bm{a}})$ is multi-transitive with respect to the vector $\bm{a}'$. Therefore, $(X^r,f_{1,\infty}^{\bm{a}})$ is strongly multi-transitive.

$(3) \Rightarrow (4)$. Let $k,m\in\mathbb{N}$ and $\bm{a}'=(1,1,...,1,2,2,...,2,...,k,k,...,k)\in\mathbb{N}^{km}$. Since $(X^r,f_{1,\infty}^{[\bm{a}]})$ is strongly multi-transitive, $(X^r,f_{1,\infty}^{[\bm{a}]})$ is multi-transitive with respect to the vector $\bm{a}'$, which implies that $$(X^{km},f_{1,\infty}^{a_1}\times f_{1,\infty}^{a_1}\times\cdots f_{1,\infty}^{a_1}\times f_{1,\infty}^{2a_1}\times f_{1,\infty}^{2a_1}\times\cdots\times f_{1,\infty}^{2a_1}\times\cdots\times f_{1,\infty}^{ma_1}\times\cdots\times f_{1,\infty}^{ma_1})$$ is transitive. Thus, $(X^{k},f_{1,\infty}\times f_{1,\infty}^{[2]}\times\cdots f_{1,\infty}^{[k]})$ is weakly mixing of all orders.

$(4) \Rightarrow (1)$. Let $l\in\mathbb{N}$ and $\bm{a}'=(a_1',a_2',...,a_l')\in\mathbb{N}^l$. Without loss of generality, assume that $a_{i}'\leq a_{i+1}'$ for all $i\in\{1,2,...,l-1\}$. Put $g_{1,\infty}=f_{1,\infty}\times f_{1,\infty}^{[2]}\times\cdots f_{1,\infty}^{[a_l']}$. Since $(X^{a_l'},g_{1,\infty})$ is weakly mixing of all orders, then $(X^{la_l'},g_{1,\infty}^{[l]})$ is transitive, which implies that $(X^{l},f_{1,\infty}^{[\bm{a}']})$ is transitive. Thus, $(X,f_{1,\infty})$ is multi-transitive with respect to the vector $\bm{a}'$. Therefore, $(X,f_{1,\infty})$ is strongly multi-transitive.
\end{proof}

\vspace{0.5cm}
\begin{theorem}\label{T4}
Let $(X,f_{1,\infty})$ be an NDS. If $(X,f_{1,\infty})$ is mildly mixing, then for every $\bm{a}$ = $(a_1, a_2, \ldots, a_r)$ $\in$ $\mathbb{N}^r$, the system $(X^{r},f_{1,\infty}^{[\bm{a}]})$ is also mildly mixing. In particular, $(X,f_{1,\infty})$ is strongly multi-transitive.
\end{theorem}

\begin{proof}
We prove this result by mathematical induction on length $r$ of $\bm{a}$.
When $r$ = 1, we will show that for any $k\in\mathbb{N}$, $(X,f_{1,\infty}^{[k]})$ is also mildly mixing.
Fixing a natural number $k$ and let $(Y,g_{1,\infty})$ be a transitive system.
 We need to show that $(X\times Y,f_{1,\infty}^{[k]}\times g_{1,\infty})$ is transitive.
Let $U_1, U_2$ be two non-empty open subsets of $X$ and $V_1, V_2$ be two non-empty open subsets of $Y$.
Let $y$ is a transitive point of $(Y,g_{1,\infty})$.
Let $K=\{0,1,2,...,k-1\}$ with discrete topology and $Z= Y\times K$.
Define $l_{m} \ : \ Z \rightarrow Z$ by
\[l_{m}(x,i)=\begin{cases}
(x,i+1), &  i=0,1, \ldots,k - 2, \\
(g_{m}(x),0), &  i=k+1,
\end{cases}\]
where $m\in\mathbb{N}$.
Define $$h_{1,\infty}=\{\underbrace{l_1,l_1,...,l_1}_{k-fold},\underbrace{l_2,l_2,...,l_2}_{k-fold},...,\underbrace{l_n,l_n,...,l_n}_{k-fold},...\}.$$
It is easy to see that $(Z,h_{1,\infty})$ is a transitive system $(y, 0)$ as a transitive point of $(Z,h_{1,\infty})$. Put $W_1=V_1\times\{0\}$ and $W_2=V_2\times\{0\}$. Since $(X,f_{1,\infty})$ is mildly mixing, $(X\times Z,f_{1,\infty}\times h_{1,\infty})$ is transitive. Then there exists an $n\in\mathbb{N}$ such that
$$f_1^n(U_1)\cap U_2\neq\emptyset \ \text{  and  }\ h_1^n(W_1)\cap W_2\neq\emptyset.$$
By the construction of $h_{1,\infty}$, we have $k\mid n$. Let $m=n/k$. Then we have
$$f_1^{km}(U_1)\cap U_2\neq\emptyset \ \text{  and  }\ g_1^m(V_1)\cap V_2\neq\emptyset,$$
which imply that $(X\times Y,f_{1,\infty}^{[k]}\times g_{1,\infty})$ is transitive. Therefore, $(X,f_{1,\infty}^{[k]})$ is mildly mixing.

For the inductive step, let $r>1$ be an integer, and assume that the result holds for $r-1$. Let $\bm{a}$ = $(a_1, a_2, \ldots, a_r)$ be a vector in $\mathbb{N}^r$ with length $r$. By the induction hypothesis, we have that $(X^{r-1},f_{1,\infty}^{[\bm{a}']})$ is mildly mixing, where $\bm{a}'$ = $(a_1, a_2, \ldots, a_{r-1})$. Let $(Y,g_{1,\infty})$ be a transitive system. Since $(X,f_{1,\infty}^{[a_r]})$ is mildly mixing, $(X\times Y,f_{1,\infty}^{[a_r]}\times g_{1,\infty})$ is transitive. Again, since $(X^{r-1},f_{1,\infty}^{[\bm{a}']})$ is mildly mixing, then $(X^{r-1}\times X\times Y,f_{1,\infty}^{[\bm{a}']})\times f_{1,\infty}^{[a_r]}\times g_{1,\infty})$ is transitive, that is, $(X^{r}\times Y,f_{1,\infty}^{[\bm{a}]})\times g_{1,\infty})$ is transitive, which implies that $(X^{r},f_{1,\infty}^{[\bm{a}]})$ is mildly mixing. Thus, the result holds for $r$, and this completes the proof.

\end{proof}

A straightforward consequence of the previous result is the following corollary.
\begin{corollary}\label{co2}
Let $(X,f_{1,\infty})$ be an NDS. If $(X,f_{1,\infty})$ is mildly mixing, then $(X,f_{1,\infty})$ is multi-transitive.
\end{corollary}

\begin{remark}
Recently, Salman and Das \cite{wx1} posed the following question:``For ADS, it is proved that if $(X, f)$ is mildly mixing, then it is multi-transitive. Does the similar result hold for NDS?"
The Corollary \ref{co2} gives an affirmative answer to the problem.
\end{remark}

The above Theorem \ref{T4} also show that if $(X,f_{1,\infty})$ is mildly mixing then $f_{1,\infty}^{[n]}$ is also mildly mixing. However, the following example will show that the converse is not true in general. In other words, for any $n\geq2$, there exists an NDS $(X,f_{1,\infty})$ such that
 $f_{1,\infty}^{[n]}$ is mildly mixing but $f_{1,\infty}$ is not mildly mixing.

\begin{example}
Let $(X,f)$ be a mixing ADS, where $X$ is a compact metric space without isolated point.
Take
$$f_{1,\infty}=\{\underbrace{id,id,...,id}_{(n-1)-fold},f,f^{-1},\underbrace{id,id,...,id}_{(n-2)-fold},f^2,f^{-2},...,\underbrace{id,id,...,id}_{(n-2)-fold},f^m,f^{-m},...\}.$$
It is easy to see that $f_1^{kn}=f^k$ for any $k\in\mathbb{N}$. Since $(X,f)$ is mixing, $f_{1,\infty}^{[n]}$ is also mixing, further, by Lemma \ref{yinli1}, $f_{1,\infty}^{[n]}$ is mildly mixing.
Take
$$g_{1,\infty}=\{f,f^{-1},\underbrace{id,id,...,id}_{(n-2)-fold},f^2,f^{-2},\underbrace{id,id,...,id}_{(n-2)-fold},...,f^m,f^{-m},\underbrace{id,id,...,id}_{(n-2)-fold},...\}.$$
With the similar argument, $g_{1,\infty}$ is transitive.
Let $U,V$ be two non-empty open sets of $X$ with $U\cap V\neq\emptyset$. For any $n\in\mathbb{N}$, it is not difficult to verify that either $f_1^n(U)\cap V\neq\emptyset$ or $g_1^n(U)\cap V\neq\emptyset$, which implies that $(X\times Y,f_{1,\infty}\times g_{1,\infty})$ can not be transitive. Therefore, $f_{1,\infty}$ is not mildly mixing.
\end{example}

\begin{theorem}\label{dingli4}
Let $n\in\mathbb{N}$ and $A\subset\mathbb{N}$ be infinite. Then the following statements are equivalent for $(X,f_{1,\infty})$:

(1) if $U_0,U_1,...,U_m\subset X$ are non-empty open sets, there exists an $n\in A$ such that $\bigcap_{i=0}^mf_1^{-in}(U_i)\neq\emptyset$;

(2) there exists a dense $G_\delta$ subset $Y\subset X$ such that for every $x\in Y$, the set $\{(f_1^{n}(x),f_1^{2n}(x),...,f_1^{mn}(x)) \mid n\in A\}$ is dense in $X^m$.

Moreover, if $A=\mathbb{N}$ and $f_{1,\infty}$ is commutative and surjective, then the statements (1) and (2) are equivalent to the following:

(3) there exists an $x\in X$ such that $\{(f_1^{n}(x),f_1^{2n}(x),...,f_1^{mn}(x):n\in \mathbb{N}\}$ is dense in $X^m$.

\end{theorem}

\begin{proof}
$(1)\Rightarrow(2)$. Suppose that $\{B_k\}_{k=1}^{\infty}$ is a countable base of $X$. Put
$$Y=\bigcap_{(k_1,...,k_m)\in\mathbb{N}^m}\bigcup_{n\in A}\bigcap_{i=1}^mf_1^{-in}(B_{k_i}).$$
By (1), it is easy to check that $\bigcup_{n\in A}\bigcap_{i=1}^mf_1^{-in}(B_{k_i})$ is open and dense in $X$. Thus by the Baire category theorem, $Y$ is a dense $G_\delta$ subset of $X$. From the construction of $Y$, for any $x\in Y$, $\{(f_1^{n}(x),f_1^{2n}(x),...,f_1^{mn}(x))\mid n\in A\}$ is dense in $X^m$.

$(2)\Rightarrow(1)$. Let $U_0,U_1,...,U_m\subset X$ be non-empty open sets. Take $x\in Y\cap U_0$. Then by (2), $\{(f_1^{n}(x),f_1^{2n}(x),...,f_1^{mn}(x):n\in A\}$ is dense in $X^m$. Hence, for $U_0,U_1,...,U_m$, there exists an $n\in A$ such that $(f_1^{n}(x),f_1^{2n}(x),...,f_1^{mn}(x))\in U_1\times U_2\times\cdots \times U_m$, which implies $x\in \bigcap_{i=0}^mf_1^{-in}(U_i)$.

Now suppose that $A=\mathbb{N}$ and $f_{1,\infty}$ is commutative.

$(2)\Rightarrow(3)$ is obvious.

Next we prove $(3)\Rightarrow(1)$. By the assumption, there exists an $x\in X$ such that $\{(f_1^{n}(x),f_1^{2n}(x),...,f_1^{mn}(x):n\in \mathbb{N}\}$ is dense in $X^m$. Let $U_0,U_1,...,U_m\subset X$ be non-empty open sets. Take $k\in \mathbb{N}$ such that $y=f_1^k(x)\in U_0$.  By the commutative of $f_{1,\infty}$, $f_{1,\infty}\times f_{1,\infty}^{[2]}\times\cdots \times f_{1,\infty}^{[m]}$ commutes with
$f_{1,\infty}\times f_{1,\infty}\times\cdots \times f_{1,\infty}$. This and the surjection of $f_{1,\infty}$ imply that the set $\{(f_1^{n}(y),f_1^{2n}(y),...,f_1^{mn}(y):n\in \mathbb{N}\}$ is also dense in $X^m$. Thus, $(f_1^{n}(y),f_1^{2n}(y),...,f_1^{mn}(y))\in U_1\times U_2\times\cdots \times U_m$ for some $n\in \mathbb{N}$. Therefore, $y\in \bigcap_{i=0}^mf_1^{-in}(U_i)$.
The proof is end.
\end{proof}

\begin{theorem}
Let $(X,f_{1,\infty})$ be an NDS and $n\in\mathbb{N}$. Then
 $f_{1,\infty}$ is $\Delta$-transitive if and only if $f_{1,\infty}^{[n]}$ is $\Delta$-transitive.
\end{theorem}
\begin{proof}
Necessity. Let $m\in\mathbb{N}$ and $U_0,U_1,U_2,\ldots,U_m$ be non-empty open subsets of $X$. For each $i\in\{0,1,2,...,m\}$ and each $1\leq j\leq nm$ with $j\neq in$, take
$$U_{in}^{'} = U_i, \ \ U_{j}^{'} = U_1.$$
Since $f_{1,\infty}$ is $\Delta$-transitive, by Theorem \ref{dingli4}, there exists an $l\in\mathbb{N}$ such that
$\bigcap_{i=0}^{nm}f_1^{-il}(U_i^{'})\neq\emptyset$, which implies that $\bigcap_{j=0}^mf_1^{-jnl}(U_j^{'})\neq\emptyset$, that is, $\bigcap_{j=0}^mf_1^{-jnl}(U_j)\neq\emptyset$. Therefore, $f_{1,\infty}^{[n]}$ is $\Delta$-transitive by Theorem \ref{dingli4}.

Sufficiency.  Let $m\in\mathbb{N}$ and $U_0,U_1,U_2,\ldots,U_m$ be non-empty open subsets of $X$. Since $f_{1,\infty}^{[n]}$  is $\Delta$-transitive, there exists an $l\in\mathbb{N}$ such that $\bigcap_{j=0}^mf_1^{-jnl}(U_i)\neq\emptyset$. That is, there exists a $k=nl\in\mathbb{N}$ such that $\bigcap_{j=0}^mf_1^{-jk}(U_i)\neq\emptyset$. Therefore, by Theorem \ref{dingli4}, $f_{1,\infty}$ is $\Delta$-transitive.

\end{proof}

\begin{theorem}
Let $(X,f_{1,\infty})$ be a $k$-periodic NDS and $g=f_k\circ f_{k-1}\circ\cdots f_1$. Then the NDS $(X,f_{1,\infty})$ is strongly multi-transitive (multi-transitive with respect to a vector, $\Delta$-transitive, respectively) if and only if the ADS $(X,g)$ is strongly multi-transitive (multi-transitive with respect to a vector, $\Delta$-transitive, respectively).
\end{theorem}

\begin{proof}
Let $m\in\mathbb{N}$ and $\bm{a}=(a_1,a_2,...,a_m)\in\mathbb{N}^m$. Let $U_1,U_2,...,U_m; V_1,V_2,...,V_m$ be a collection of non-empty open subsets of $X$. Since $(X,g)$ is strongly multi-transitive, there exists an $l\in\mathbb{N}$ such that $g^{a_jl}(U_j)\cap V_j\neq\emptyset$ for every $j\in\{1,2,...,m\}$. By the definition of $g$, we have $f_1^{ka_jl}(U_j)\cap V_j\neq\emptyset$ for every $j\in\{1,2,...,m\}$. Therefore, $(X,f_{1,\infty})$ is strongly multi-transitive.

Conversely, let $m\in\mathbb{N}$ and $\bm{a}=(a_1,a_2,...,a_m)\in\mathbb{N}^m$. And let $U_1,U_2,...,U_m; V_1,V_2,...,V_m$ be a collection of non-empty open subsets of $X$. Put
$$U_{i}^{'} = U_1, U_{k+i}^{'} = U_2, \ldots, U_{(m-1)k+i}^{'} = U_m,$$
$$V_{i}^{'} = V_1, V_{k+i}^{'} = V_2, \ldots, V_{(m-1)k+i}^{'} = V_m,$$
where  $i=1,2,...,k$. Let $$\bm{a}'=(b_1,b_2,...,b_{mk})=(a_1,2a_1,...,ka_1,a_2,2a_2,...ka_2,...,a_m,2a_m,...,ka_m)\in\mathbb{N}^{mk}.$$ Since $(X,f_{1,\infty})$ is strongly multi-transitive, then it is multi-transitive with respect to the vector $\bm{a}'$, therefore $(X^{mk},f_{1,\infty}^{[\bm{a}']})$ is transitive and hence there exists an $s\in\mathbb{N}$ such that $f_1^{b_js}(U_j^{'})\cap V_j^{'}\neq\emptyset$ for every $j\in\{1,2,...,km\}$. By $k$-periodicity of $(X,f_{1,\infty})$, we have $f_1^{ka_js}(U_j)\cap V_j\neq\emptyset$ for every $j\in\{1,2,...,m\}$. Equivalently, $g^{a_js}(U_j)\cap V_j\neq\emptyset$ for every $j\in\{1,2,...,m\}$. Therefore, $(X,g)$ is strongly multi-transitive.

The proofs of other cases are similar, so we omit them.

\end{proof}

\begin{theorem}
Let $(X,f_{1,\infty})$ and $(Y,g_{1,\infty})$ be two NDS such that $f_{1,\infty}$ is topologically semi-conjugate to $g_{1,\infty}$. If $f_{1,\infty}$ is strongly multi-transitive (multi-transitive with respect to a vector, $\Delta$-mixing, $\Delta$-transitive, respectively), then $g_{1,\infty}$ is also strongly multi-transitive (multi-transitive with respect to a vector, $\Delta$-mixing, $\Delta$-transitive, respectively).
\end{theorem}

\begin{proof}
Let $m\in\mathbb{N}$ and $\bm{a}=(a_1,a_2,...,a_m)\in\mathbb{N}^m$. Let $U_1,U_2,...,U_m$,$V_1,V_2,...,V_m$ be a collection of non-empty open subsets of $Y$. Since $f_{1,\infty}$ is topologically semi-conjugate to $g_{1,\infty}$, therefore each of $h^{-1}(U_i)$ and $h^{-1}(V_i)$ is non-empty open subset of $X$ for any $i\in\{1,2,...,m\}$. Since $f_{1,\infty}$ is strongly multi-transitive, there exists an $l\in\mathbb{N}$ such that $f_1^{a_jl}(h^{-1}(U_j))\cap h^{-1}(V_j)\neq\emptyset$ for every $j\in\{1,2,...,m\}$. Now, as $h\circ f_k=g_k\circ h$ for any $k\in\mathbb{N}$, one can get that for any $n\in\mathbb{N}$ $h\circ f_1^n=g_1^n\circ h$. Further, we have $f_1^n(h^{-1}(A))\subseteq h^{-1}\circ g_1^n(A)$ for any  $A\subseteq X$ and any $n\in\mathbb{N}$. Hence, for every $j\in\{1,2,...,m\}$,
$$\emptyset\neq f_1^{a_jl}(h^{-1}(U_j))\cap h^{-1}(V_j)\subseteq h^{-1}\circ g_1^{a_jl}(U_j)\cap h^{-1}(V_j)=h^{-1}(g_1^{a_jl}(U_j)\cap V_j).$$
Therefore, $g_1^{a_jl}(U_j)\cap V_j\neq\emptyset$ for every $j\in\{1,2,...,m\}$. Thus, $g_{1,\infty}$ is strongly multi-transitive.

With the similar argument, the result also holds for multi-transitive with respect to a vector ($\Delta$-mixing, $\Delta$-transitive, respectively).
\end{proof}

For ADS it is known that if $f,g$ are both syndetically transitive and weakly mixing, then $f\times g$ is syndetically transitive and weakly mixing. However, the result is not always true for NDS as justified by the next example.

\begin{example}\label{lizi2}

Let $f$  be a mixing homeomorphism map on $X$. Consider the non-autonomous discrete systems $(X,f_{1,\infty})$ and $(X,g_{1,\infty})$, where $f_{1,\infty}$ is defined by
$$f_{1,\infty}=\{f,f^{-1},f^2,f^{-2},...,f^n,f^{-n},...\},$$
and $g_{1,\infty}$ is defined by
$$g_{1,\infty}=\{id,f,f^{-1},f^2,f^{-2},...,f^n,f^{-n},...\}.$$
Since $f$ is mixing and $f_1^{2k-1}=g_1^{2k}=f^k$ for each $k\in\mathbb{N}$, $f_{1,\infty}$ and $g_{1,\infty}$ are both syndetically transitive and weakly mixing.  But note that $f_1^{2k}=f_1^{2k-1}=id$ for any $k\in\mathbb{N}$, so for any $n\in\mathbb{N}$ and any non-empty open subsets of $U,V$ of $X$ with $U\cap V\neq\emptyset$, either $f_1^n(U)\cap V\neq\emptyset$ or $g_1^n(U)\cap V\neq\emptyset$, which implies that $f_{1,\infty}\times g_{1,\infty}$ is neither syndetically transitive nor weakly mixing.

\end{example}

\begin{remark}
For ADS it is known that if $(X,f)$ is syndetically transitive and weakly mixing, then it is  thickly syndetically transitive. However, the above example shows that the result is not always true for NDS even though $f_{1,\infty}$ is syndetically transitive and weakly mixing of all orders.
\end{remark}

For ADS it is known that weakly mixing is equivalent to the thick transitivity. Therefore, syndetical transitivity and thick transitivity imply thickly syndetical transitivity in ADS. However, the result does not hold in NDS. The following example shows that there exists an NDS such that $f_{1,\infty}$ is syndetically and thickly  transitive but is not thickly syndetically transitive.

\begin{example}\label{li3}

Let $f$  be a mixing homeomorphism map on $X$.

Let $$h_i^{10}=\{f^i,f^{-i}, \underbrace{id,id,...,id}_{8-fold}\},$$
$$g_i^{k+9}=\{\underbrace{f^i,f^i,...,f^i}_{k-fold},f^{-ki},\underbrace{id,id,...,id}_{8-fold}\},$$
where $i,k\in\mathbb{N}$.

Consider the NDS $(X,f_{1,\infty})$, where $f_{1,\infty}=\{f_1,f_1,f_3,...,f_n,...\}$ is defined by
\[f_n^{p}=\begin{cases}
g_{k}^{k+9}, &  \text{if } n\in[10^k,10^k+10), \\
h_k^{10}, &  else,  \\
\end{cases}\]
where $n,k\in\mathbb{N}$ and
\[p=\begin{cases}
k+9, &  \text{if } n\in[10^k,10^k+10), \\
10, &  else.  \\
\end{cases}\]
That is,
$$f_{1,\infty}=\{f_1,f_1,f_3,...,f_n,...\}=\{h_1^{10},g_1^{10},\underbrace{h_1^{10},h_1^{10},...,h_1^{10}}_{8-fold}, g_2^{11},\underbrace{h_2^{10},h_2^{10},...,h_2^{10}}_{89-fold},g_3^{12},h_3^{10},...\}.$$
Since $f$ is mixing, then for any two non-empty open sets $U,V\subset X$, there exists an $M\in\mathbb{N}$ such that $f_1^n(U)\cap V\neq\emptyset$ for each $n>M$. Thus, it is easy to see that $f_{1,\infty}$ is thickly transitive and that $M+10$ is bounded gaps of $N_{f_{1,\infty}}(U,V)$. Therefore, $f_{1,\infty}$ is also syndetically transitive.  But take $j=10$,  for any syndetic set $A\subset\mathbb{N}$, there exist two positive integers $a\in A,b\in[0,10]$ such that $f_1^{a+b}=id$. Therefore, $f_{1,\infty}$ is not thickly syndetically transitive.

\end{example}

\begin{remark}
In \cite[Example 3.1]{wx1}, the authors gave an example showing that there exists an NDS which is weakly mixing and syndetically transitive but not multi-transitive. The Example \ref{li3} also shows that there exists an NDS which is thickly and syndetically transitive but fail to be multi-transitive. However, for ADS, thick transitivity and syndetical transitivity imply multi-transitivity \cite[Corollary 2]{wx2}. The \cite[Example 3.1]{wx1} also showed that there exists an NDS which is weakly mixing of all orders but not thickly transitive. The example below will show that there exists an NDS which is thickly transitive but not weakly mixing.
\end{remark}

\begin{example}\label{lizi4}

Let $f:S^1\rightarrow S^1$ be the map defined by $f(e^{2\pi i\theta})=e^{2\pi i(\theta+\alpha)}$, where $S^1$ is a unit circle on the complex plane, $\theta\in [0,1]$ and $\alpha$ is a irrational number. Consider the non-autonomous discrete systems $(S,f_{1,\infty})$, where $f_{1,\infty}$ is  given by
$$g_{1,\infty}=\{f,id,f,id,id,f,id,id,id,f,id,id,id,id,...,f,\underbrace{id,id,...id}_{n-fold},...\}.$$

Note that $orb_{f_{1,\infty}}(x)=orb_f(x)$. Since $f$ is transitive but not weakly mixing, then it is easy to see that $(S,f_{1,\infty})$ is thickly transitive but not weakly mixing.

\end{example}

The Example \ref{lizi2}, Example \ref{li3} and Example \ref{lizi4} show that the results related to stronger forms of transitivity which are true for ADS but fail in NDS, which show that there is a significant difference between the theory of ADS and the theory of NDS. Next we will establish a sufficient condition under which the results still hold in NDS.

Firstly, we give a sufficient condition under which the syndetical transitivity and weakly mixing of $f_{1,\infty}$ and $g_{1,\infty}$ imply the syndetical transitivity and weakly mixing of $f_{1,\infty}\times g_{1,\infty}$.
\begin{theorem}
Let $(X,f_{1,\infty}),(Y,g_{1,\infty})$ be an NDS with $f_{1,\infty},g_{1,\infty}$  being feeble open. Suppose that $\{f_n^{n+k}\}_{k\in\mathbb{N}}$ converges collectively to $\{f^k\}_{k\in\mathbb{N}}$ and that  $\{g_n^{n+k}\}_{k\in\mathbb{N}}$ converges collectively to $\{g^k\}_{k\in\mathbb{N}}$. If $f_{1,\infty},g_{1,\infty}$ are both syndetically transitive and weakly mixing, then $f_{1,\infty}\times g_{1,\infty}$ is syndetically transitive and weakly mixing.

\end{theorem}

\begin{proof}
By \cite[Theorem 3.1]{wx6}, \cite[Proposition 4]{wx2} and \cite[Corollary 1]{q36}, we get the result.

\end{proof}

Then, we give a sufficient condition under which the syndetical transitivity and weakly mixing (or thickly transitive) of $f_{1,\infty}$ imply the thickly syndetical transitivity and multi-transitivity of $f_{1,\infty}$.
\begin{theorem}\label{budongdian0}
Let $(X,f_{1,\infty})$ be an NDS with $f_{1,\infty}$  being feeble open. Suppose that $\{f_n^{n+k}\}_{k\in\mathbb{N}}$ converges collectively to $\{f^k\}_{k\in\mathbb{N}}$. If $f_{1,\infty}$ is syndetically transitive and weakly mixing (or thickly transitive), then $f_{1,\infty}$ is thickly syndetically transitive and multi-transitive.

\end{theorem}

In order to show this result we need the following two theorems. Firstly, using the proof of \cite[Theorem 4.2]{wx1}, we have the result below.
\begin{theorem}\label{thick}
Let $(X,f_{1,\infty})$ be a non-autonomous system with $f_{1,\infty}$ being feeble open and surjective and converging uniformly to a map $f$. If $\{f_n^{n+k}\}_{k\in\mathbb{N}}$ converges collectively to $\{f^k\}_{k\in\mathbb{N}}$, then $f_{1,\infty}$ is thickly transitive if and only if $f$ is thickly transitive.
\end{theorem}

\begin{theorem}\label{thsy}
Let $(X,f_{1,\infty})$ be a non-autonomous system with $f_{1,\infty}$ being feeble open and converging uniformly to a map $f$. If $\{f_n^{n+k}\}_{k\in\mathbb{N}}$ converges collectively to $\{f^k\}_{k\in\mathbb{N}}$, then $f_{1,\infty}$ is thickly syndetically transitive if and only if $f$ is thickly syndetically transitive.
\end{theorem}
\begin{proof}
Let $U,V$ be two non-empty open sets in $X$. Then there are $\epsilon>0$ and $x,y\in X$ such that $B(x,\epsilon)\subseteq U$ and $B(y,\epsilon)\subseteq V$. As $\{f_n^{n+k}\}_{k\in\mathbb{N}}$ converges collectively to $\{f^k\}_{k\in\mathbb{N}}$, there exists $r_0\in \mathbb{N}$ such that $D(f_r^{r+k},f^k) < \frac{\epsilon}{2}$ for all $r\geq r_0$ and for every $k\in\mathbb{N}$. Since $f_{1,\infty}$ is feeble open, $int(f_1^{r_0}(U))$ is non-empty open subset of $X$. By thickly syndetical transitivity of $(X,f)$, for open sets $U'=int(f_1^{r_0}(U))$ and $V'=B(y,\frac{\epsilon}{2})$, the set $N_f(U',V')$ is thickly syndetical. Take  $m\in N_f(U',V')$. Then $f^{m}(U')\cap V' \neq\emptyset$. Consequently, there exist $u'\in U'$ such that $f^{m}(u')\in V'$. Further, as $U'=int(f_1^{r_0}(U))$, there exist $u\in U$ such that $u'=f_1^{r_0}(u)$, hence $f^{m}(f_1^{r_0}(u))\in V'$.
Also, collective convergence ensures $d(f_1^{m+r_0}(u),f^{m}(f_1^{r_0}(u)))<\frac{\epsilon}{2}$, and hence by the triangle inequality $d(y,f_1^{m+r_0}(u))<\epsilon$. Therefore, $f_1^{m+r_0}(U)\cap V\neq\emptyset$, which implies that $N_f(U',V')+r_0\subseteq N_{f_{1,\infty}}(U,V)$. Therefore, $N_{f_{1,\infty}}(U,V)$ is thickly syndetical and hence $(X,f_{1,\infty})$ is thickly syndetically transitive.

Conversely, let $U,V$ be two non-empty open sets in $X$. Then there are $\epsilon>0$ and $x,y\in X$ such that $B(x,\epsilon)\subseteq U$ and $B(y,\epsilon)\subseteq V$.  As $\{f_n^{n+k}\}_{k\in\mathbb{N}}$ converges collectively to $\{f^k\}_{k\in\mathbb{N}}$, there exists $r_0\in \mathbb{N}$ such that $D(f_r^{r+k},f^k) < \frac{\epsilon}{2}$ for all $r\geq r_0$ and for every $k\in\mathbb{N}$. Put $U_1=f_1^{-r_0}(B(x,\epsilon))$, $V_1=B(y,\frac{\epsilon}{2})$. By thickly syndetical transitivity of $(X,f_{1,\infty})$, the set $N_{f_{1,\infty}}(U_1,V_1)$ is thickly syndetical. Let $k\in N_{f_{1,\infty}}(U_1,V_1)-r_0$. Then $f_1^{r_0+k}(U_1)\cap V_1\neq\emptyset$. Consequently, there exist $u\in U_1$ such that $d(f_1^{r_0+k}(u),y)<\frac{\epsilon}{2}$. Further, collective convergence ensures $d(f_1^{r_0+k}(u),f^{k}(f_1^{r_0}(u)))<\frac{\epsilon}{2}$, and hence by triangle inequality we have $d(y,f^{k}(f_1^{r_0}(u)))<\epsilon$. As $f_1^{r_0}(u)\in B(x,\epsilon)$, we have $f^{k}(B(x,\epsilon))\cap B(y,\epsilon)\neq\emptyset$, which implies that $N_f(U_1,V_1)-r_0\subseteq N_f(U,V)$. Therefore, $N_f(U,V)$ is thickly syndetical and hence the system $(X,f)$ is thickly syndetically transitive.

\end{proof}

\noindent {\textit{Proof of Theorem \ref{budongdian0}}}. By \cite[Theorem 3.1]{wx6}, \cite[Lemma 1]{wx2}, \cite[Corollary 1]{q36} (or Theorem \ref{thick}), \cite[Corollary 2]{wx2}, \cite[Theorem 3.1]{wx1} and Theorem \ref{thsy}, we get the result.

\qed

For ADS it is known that weakly mixing implies total transitivity. However, the \cite[Example 3.1]{wx1} shows that the result is not always true for NDS. But the following theorem shows that if we replace weakly mixing by thick transitivity, then the result also holds for NDS.

\begin{theorem}
Let $(X,f_{1,\infty})$ be an NDS. If $f_{1,\infty}$ is thickly transitive, then it is totally transitive.
\end{theorem}

\begin{proof}
Let $U,V$ be two non-empty open subsets of $X$ and $n\in\mathbb{N}$. Since $f_{1,\infty}$ is thickly transitive, then there exists an $m\in\mathbb{N}$ such that $\{m,m+1,...,m+n\}\subset N_{f_{1,\infty}}(U,V)$. For $\{m,m+1,...,m+n\}$, there exists an $k\in\{m,m+1,...,m+n\}$ such that $n\mid k$, that is, $k=nj$, where $j\in\mathbb{N}$. Therefore, $f_1^{nj}(U)\cap V=f_1^{k}(U)\cap V\neq\emptyset$, which implies that $f_{1,\infty}^{[n]}$ is  transitive. Therefore, $f_{1,\infty}$ is totally transitive.

\end{proof}

\begin{theorem}\label{strong}
Let $(X,f_{1,\infty})$ be a non-autonomous system with $f_{1,\infty}$ being feeble open and converging uniformly to a map $f$. If $\{f_n^{n+k}\}_{k\in\mathbb{N}}$ converges collectively to $\{f^k\}_{k\in\mathbb{N}}$, then $f_{1,\infty}$ is strongly multi-transitive if and only if $f$ is strongly multi-transitive.
\end{theorem}

\begin{proof}
Let $l\in\mathbb{N}$ and $\bm{a}=(a_1,a_2,...,a_l)\in\mathbb{N}^l$.  Let $\epsilon>0$ be given and let $U_i=B(x_i,\epsilon)$ and $V_i=B(y_i,\epsilon)$ for $i\in\{1,2,...,l\}$ be a collection of non-empty open sets in $X$. As $\{f_n^{n+k}\}_{k\in\mathbb{N}}$ converges collectively to $\{f^k\}_{k\in\mathbb{N}}$, there exists $r_0\in \mathbb{N}$ such that $D(f_r^{r+k},f^k) < \frac{\epsilon}{2}$ for all $r\geq r_0$ and for every $k\in\mathbb{N}$. Since $f_{1,\infty}$ is feeble open, each $int(f_1^{a_ir_0}(U_i))$ is non-empty open subset of $X$ for $i\in\{1,2,...,l\}$. Thus by strongly multi-transitivity of $(X,f)$, for open sets $U_i'=int(f_1^{a_ir_0}(U_i))$ and $V_i'=B(y_i,\frac{\epsilon}{2})$, there exists $m\in \mathbb{N}$ such that $f^{ma_i}(U_i')\cap V_i' \neq\emptyset$ for every $i\in\{1,2,...,l\}$. Consequently there exist $u_i'\in U_i'$ such that $f^{ma_i}(u_i')\in V_i'$ for every $i\in\{1,2,...,l\}$. Further, as $U_i'=int(f_1^{a_ir_0}(U_i))$, there exist $u_i\in U_i$ such that $u_i'=f_1^{a_ir_0}(u_i)$ for every $i\in\{1,2,...,l\}$, hence $f^{ma_i}(f_1^{a_ir_0}(u_i))\in V_i'$ for every $i\in\{1,2,...,l\}$.
Also, collective convergence ensures $d(f_1^{(m+r_0)a_i}(u_i),f^{ma_i}(f_1^{a_ir_0}(u_i)))<\frac{\epsilon}{2}$, and hence by the triangle inequality $d(y_i,f_1^{(m+r_0)a_i}(u_i))<\epsilon$ for every $i\in\{1,2,...,l\}$. Therefore, $f_1^{(m+r_0)a_i}(U_i)\cap V_i\neq\emptyset$ for every $i\in\{1,2,...,l\}$. Hence $(X,f_{1,\infty})$ is strongly multi-transitive.

Conversely, let $l\in\mathbb{N}$ and $\bm{a}=(a_1,a_2,...,a_l)\in\mathbb{N}^l$.  Let $\epsilon>0$ be given and let $U_i=B(x_i,\epsilon)$ and $V_i=B(y_i,\epsilon)$ for $i\in\{1,2,...,l\}$ be a collection of non-empty open sets in $X$. As $\{f_n^{n+k}\}_{k\in\mathbb{N}}$ converges collectively to $\{f^k\}_{k\in\mathbb{N}}$, there exists $r_0\in \mathbb{N}$ such that $D(f_r^{r+k},f^k) < \frac{\epsilon}{2}$ for all $r\geq r_0$ and for every $k\in\mathbb{N}$. Put $U_i=f_1^{-r_0a_i}(B(x_i,\epsilon))$, $V_i=S(y_1,\frac{\epsilon}{2})$ for every $i\in\{1,2,...,l\}$. Applying Lemma \ref{yinli2} we can choose $k\in\mathbb{N}$ such that $f_1^{(r_0+k)a_i}(U_i)\cap V_i\neq\emptyset$ for every $i\in\{1,2,...,l\}$. Consequently, for every $i\in\{1,2,...,l\}$, there exist $u_i\in U_i$ such that $d(f_1^{(r_0+k)a_i}(u_i),y_i)<\frac{\epsilon}{2}$. Further, collective convergence ensures $d(f_1^{(r_0+k)a_i}(u_i),f^{ka_i}(f_1^{r_0a_i}(u_i)))<\frac{\epsilon}{2}$, and hence by triangle inequality we have $d(y_i,f^{ka_i}(f_1^{r_0a_i}(u_i)))<\epsilon$ for every $i\in\{1,2,...,l\}$. As $f_1^{a_ir_0}(u_i)\in B(x_i,\epsilon)$, we have $f^{ka_i}(B(x_i,\epsilon))\cap B(y_i,\epsilon)\neq\emptyset$ for every $i\in\{1,2,...,l\}$. Therefore, the system $(X,f)$ is strongly multi-transitive.

\end{proof}

\begin{theorem}\label{dengjia6}
Let $(X,f_{1,\infty})$ be a non-autonomous system with $f_{1,\infty}$ being feeble open and surjective and converging uniformly to a map $f$. If $(X,f_{1,\infty})$ is minimal such that $\{f_n^{n+k}\}_{k\in\mathbb{N}}$ converges collectively to $\{f^k\}_{k\in\mathbb{N}}$, then the following are equivalent:

(1)$f_{1,\infty}\times f_{1,\infty}^{[2]}: X^2\rightarrow X^2$ is transitive.

(2)$f_{1,\infty}$ is multi-transitive.

(3)$f_{1,\infty}$ is weakly mixing.

(4)$f_{1,\infty}$ is weakly mixing of all orders.

(5)$f_{1,\infty}$ is thickly transitive.

(6)$f_{1,\infty}$ is strongly multi-transitive.

\end{theorem}

\begin{proof}
By \cite[Corollary 4.1]{wx1}, (1), (2), (4) and (5) are equivalent. By \cite[Corollary 1]{q36} and the fact that $f$ is weakly mixing if and only if $f$ is weakly mixing of all orders for ADS, we get that (3) is equivalent to (4). By Theorem \ref{strong4} and Theorem \ref{strong}, we have that (6) is equivalent to (4).

\end{proof}

\begin{theorem}\label{delta}
Let $(X,f_{1,\infty})$ be an NDS and $\{f_n^{n+k}\}_{k\in\mathbb{N}}$ converge collectively to $\{f^k\}_{k\in\mathbb{N}}$. \\
(1) If $f_{1,\infty}$ is feeble open and $\Delta$-transitive, then $f$ is also $\Delta$-transitive. \\
(2) If $f$ commutes with each $f_n$ of $f_{1,\infty}$ and is $\Delta$-transitive, then $f_{1,\infty}$ is also $\Delta$-transitive.
\end{theorem}

\begin{proof}
(1). Let $m\in\mathbb{N}$ and $U_1,U_2,...,U_m\subset X$ be non-empty open sets. Then there exist $u_1,u_2,...,u_m\in X$ and $\varepsilon>0$ such that $B(u_i,\varepsilon)\subset U_i$ for each $i=1,2,...,m$. As $\{f_n^{n+k}\}_{k\in\mathbb{N}}$ converges collectively to $\{f^k\}_{k\in\mathbb{N}}$, there exists $r_0\in\mathbb{N}$ such that $D(f_r^k,f^k) < \frac{\varepsilon}{2}$ for all $r\geq r_0$ and for every $k\in\mathbb{N}$, which implies $d(f_r^k(z),f^k(z)) < \frac{\varepsilon}{2}$ for all $r\geq r_0$ and for every $k\in\mathbb{N}$ and any $z\in X$. Put $W_i=B(u_i,\frac{\varepsilon}{2})$ and $V_i=int(f_1^{ir_0}(W_i))$ for each $i=1,2,...,m$. Since $f_{1,\infty}$ is feeble open, therefore for every $i\in\{1,2,...,m\}$ $V_i$ is a non-empty open set. By $\Delta$-transitivity of $f_{1,\infty}$, there exists a dense $G_\delta$ subset $Y\subset X$ such that for every $x\in Y$, $\{(f_1^{n}(x),f_1^{2n}(x),...,f_1^{mn}(x):n\in A\}$ is dense in $X^m$. That is, for every $x\in Y$ there exists an $l>r_0$ such that $f_1^{li}(x)\in V_i$, which implies that $f_1^{li}(x)\in f_1^{ir_0}(W_i)$. Therefore, $f_{ir_0+1}^{(l-r_0)i}(x)\in W_i$, which implies $d(f_{ir_0+1}^{(l-r_0)i}(x),u_i)<\frac{\varepsilon}{2}$. Since $d(f_{ir_0+1}^{(l-r_0)i}(x),f^{(l-r_0)i}(x)) < \frac{\varepsilon}{2}$, by triangle inequality, we have $d(f^{(l-r_0)i}(x),u_i) <\varepsilon$, which implies that $f^{(l-r_0)i}(x)\in U_i$ for each $i=1,2,...,m$. Thus, $f$ is also $\Delta$-transitive.

(2). Let $m\in\mathbb{N}$ and $U_1,U_2,...,U_m\subset X$ be non-empty open sets. Then there exist $u_1,u_2,...,u_m\in X$ and $\varepsilon>0$ such that $B(u_i,\varepsilon)\subset U_i$ for each $i=1,2,...,m$. As $\{f_n^{n+k}\}_{k\in\mathbb{N}}$ converges collectively to $\{f^k\}_{k\in\mathbb{N}}$, there exists $r_0\in\mathbb{N}$ such that $D(f_r^k,f^k) < \frac{\varepsilon}{2}$ for all $r\geq r_0$ and for every $k\in\mathbb{N}$, which implies $d(f_r^k(z),f^k(z)) < \frac{\varepsilon}{2}$ for all $r\geq r_0$ and for every $k\in\mathbb{N}$ and any $z\in X$. Put $W_i=B(u_i,\frac{\varepsilon}{2})$ and $V_i=f_1^{-ir_0}(W_i)$ for each $i=1,2,...,m$. By $\Delta$-transitivity of $f$, there exists a dense $G_\delta$ subset $Y\subset X$ such that for every $x\in Y$, $\{(f_1^{n}(x),f_1^{2n}(x),...,f_1^{mn}(x):n\in A\}$ is dense in $X^m$. That is, for every $x\in Y$ there exists an $l\in\mathbb{N}$ such that $f^{li}(x)\in V_i$, which implies that $f^{li}(x)\in f_1^{-ir_0}(W_i)$. Therefore, $f_{1}^{ir_0}\circ f^{li}(x)\in W_i$, which implies $d(f_{1}^{ir_0}\circ f^{li}(x),u_i)<\frac{\varepsilon}{2}$. Since $f$ commutes with each $f_n$ of $f_{1,\infty}$, we have $d(f^{li}\circ f_{1}^{ir_0}(x),u_i)<\frac{\varepsilon}{2}$. As $d(f_{ir_0+1}^{li}\circ f_{1}^{ir_0}(x),f^{li}\circ f_{1}^{ir_0}(x)) < \frac{\varepsilon}{2}$, by triangle inequality, we have $d(f_{ir_0+1}^{li}\circ f_{1}^{ir_0}(x),u_i) <\varepsilon$, which implies that $f^{(l+r_0)i}(x)\in U_i$ for each $i=1,2,...,m$. Thus, $f_{1,\infty}$ is also $\Delta$-transitive.

\end{proof}

For autonomous discrete dynamical system it is known that if $(X,f)$ is a minimal homeomorphism system, then weakly mixing, multi-transitivity and $\Delta$-transitivity are equivalent. However, the following example shows that the result is not always true for NDS.

\begin{example}

Let $(X,f)$ be autonomous discrete dynamical system such that $f$ is a homeomorphism which is both minimal and weakly mixing (in ). Consider the non-autonomous discrete systems $(X,f_{1,\infty})$, where $f_{1,\infty}$ is  given by
$$f_{1,\infty}=\{f,f^{-1},f^2,f^{-2},...,f^n,f^{-n},...\}.$$

Since $f$ is weakly mixing and $f_1^{2k-1}=g_1^{2k}=f^k$ for each $k\in\mathbb{N}$, $f_{1,\infty}$ is weakly mixing. Note that $orb_{f_{1,\infty}}(x)=\{x,f(x),f^2(x),...\}=orb_f(x)$ for any $x\in X$. Since $f$ is  minimal, $f_{1,\infty}$ is also minimal. But note that $f_1^{2k}=id$ for any $k\in\mathbb{N}$, so $f_{1,\infty}$ is neither multi-transitive  nor $\Delta$-transitive.

\end{example}

\begin{theorem}
Let $(X,f_{1,\infty})$ be an NDS and $\{f_n^{n+k}\}_{k\in\mathbb{N}}$ converge collectively to $\{f^k\}_{k\in\mathbb{N}}$. If $f_{1,\infty}$ is feeble open and $f$ commutes with each $f_n$ of $f_{1,\infty}$. Then the following are equivalent:

(1) $f_{1,\infty}\times f_{1,\infty}^{[2]}: X^2\rightarrow X^2$ is transitive.

(2) $f_{1,\infty}$ is multi-transitive.

(3) $f_{1,\infty}$ is weakly mixing.

(4) $f_{1,\infty}$ is weakly mixing of all orders.

(5) $f_{1,\infty}$ is thickly transitive.

(6) $f_{1,\infty}$ is strongly multi-transitive.

(7) $f_{1,\infty}$ is $\Delta$-transitive.

\end{theorem}

\begin{proof}

By Theorem \ref{dengjia6}, we have that (1)-(6) are equivalent. By \cite[Corollary 7]{wx2} and Theorem \ref{delta}, we obtain that (7) is equivalent to (3).
\end{proof}

\begin{theorem}\label{deltamixing}
Let $(X,f_{1,\infty})$ be an NDS. If $f_{1,\infty}$ is chain mixing and has the shadowing property, then $f_{1,\infty}$ is $\Delta$-mixing.
\end{theorem}

\begin{proof}
Let $m\in\mathbb{N}$ and $U_0,U_1,...,U_m\subset X$ be non-empty open sets and $A\subset \mathbb{Z_+}$ be infinite. Then there exist $u_0,u_1,...,u_m\in X$ and $\varepsilon>0$ such that $B(u_i,\varepsilon)\subset U_i$ for each $i=0,2,...,m$. Let $\delta>0$ be provided to $\varepsilon$ such that each $\delta$-pseudo-orbit is $\varepsilon$-shadowed by some point in $X$. Since $f_{1,\infty}$ is chain mixing, there exists $l\in A$ such that there is a $\delta$-chain $\xi_i$ from $x_i$ to $x_{i+1}$ with length $l$ for each $i=0,2,...,m-1$. Then $\xi_0,\xi_1\setminus\{x_1\},\xi_2\setminus\{x_2\},...,\xi_{m-2}\setminus\{x_{m-2}\},\xi_{m-1}$ is a $\delta$-chain from $x_0$ to $x_m$. Hence, there exists a $z\in X$ such that for each $i=0,2,...,m$, we have $d(f_1^{il}(z),x_i)<\varepsilon$, which implies that
$$z\in\bigcap_{i=0}^mf_1^{-in}(U_i)\neq\emptyset.$$
Therefore, $f_{1,\infty}$ is $\Delta$-mixing.
\end{proof}
With the similar argument to the proof in Theorem \ref{deltamixing}, we can prove the next result.
\begin{theorem}
Let $(X,f_{1,\infty})$ be an NDS. If $f_{1,\infty}$ is chain transitive and has the shadowing property, then $f_{1,\infty}$ is transitive.
\end{theorem}

\noindent\textbf{Question 1.}  In ADS, Moothathu \cite{wx2} has proved that in a minimal system, weakly mixing and multi-transitivity are equivalent. However, for an NDS, Das has proved that in a minimal system, weakly mixing does not imply multi-transitivity\cite{wx1}. Does the inverse hold? That is, in a minimal NDS, Does multi-transitivity imply weakly mixing?

\noindent\textbf{Question 2.}  In ADS, Moothathu \cite{wx2} has proved that  $\Delta$-transitivity implies  weakly mixing. Does the result still hold in NDS?  Further, does $\Delta$-transitivity imply weakly mixing of all orders in NDS?

\section*{Acknowledgments}
The author would like to thank the the editor and the anonymous referees for their constructive comments and valuable suggestions.


\begin{thebibliography}{0}
\bibitem{q12}G. Birkhoff, Dynamical Systems, American Mathematical Society, Washington, 2008.
\bibitem{q27}E. Akin, J. Auslander, A. Nagar, Variations on the concept of topological transitivity, Studia Mathematica, 235(3)(2016)225-249.
\bibitem{wx2} T. Moothathu, Diagonal points having dense orbit, Colloquium Mathematicum, 120(2010)127-138.
\bibitem{q37} D. Kwietniak, P. Oprocha, On weak mixing, minimality and weak disjointness of all iterates, Ergodic Theory and Dynamical Systems, 32(5)(2012)1661-1672.
\bibitem{q38} Z. Chen, J. Li, J. Lv, On multi-transitivity with respect to a vector, Science China-mathematics, 57 (8) (2014) 1639-1648.

\bibitem{q39} H. Wang, H. Fu, S. Diao, P. Zeng, Chain mixing, shadowing properties and multi-transitivity, Bulletin of The Iranian Mathematical Society, 45(6)(2019)1605-1618.

\bibitem{wx3}R. Luis, S. Elaydi, H. Oliveira, Non-autonomous periodic systems with Allee effects, Journal of Difference Equations and Applications, 16 (10)(2010)1179-1196.

\bibitem{wx4}R. Vasisht, R. Das, Generalizations of expansiveness in non-autonomous discrete systems, Bulletin of the Iranian Mathematical Society, 48(2022) 417-433.
\bibitem{wx6}R. Vasisht, R. Das. On stronger forms of sensitivity in non-autonomous systems, Taiwanese Journal of Mathematics, 22(2018)1139-1159.
\bibitem{q24}I. Sanchez, M. Sanchis, H. Villanueva, Chaos in hyperspaces of nonautonomous discrete systems, Chaos Solitons and Fractals, 94(2017)68-74.
\bibitem{a14} S. Kolyada, L. Snoha, Topological entropy of nonautononous dynamical systems, Random Computer and Dynamics, 4(1996)205-233.
\bibitem{q16}C. Tian, G. Chen, Chaos of a sequence of maps in a metric space, Chaos Solitons and Fractals, 28(2006)1067-1075.


\bibitem{q18}M. Salman, R. Das, Sensitivity and property P in non-autonomous systems, Mediterranean Journal of Mathematics, 17(2020)pp.128.

\bibitem{q14} A. Miralles,  M. Murillo-Arcila,  M. Sanchis, Sensitive dependence for nonau-tonomous discrete dynamical systems, Journal of Mathematical Analysis and Applications, 463(2018)268-275.
\bibitem{q23} J. Dvorakova, Chaos in nonautonomous discrete dynamical systems, Communications in Nonlinear Science and Numerical Simulation, 17(2012)4649-4652.
\bibitem{wx1}M. Salman, R. Das, Multi-transitivity in nonautonomous discrete systems, Topologe and its Applications, 278(2020)107237.
\bibitem{q25}F. Balibrea, P. Oprocha, Weak mixing and chaos in nonautonomous discrete systems, Applied Mathematics Letters, 25(2012)1135-1141.

\bibitem{q36}P. Sharma, M. Raghav, On dynamics generated by a uniformly convergent sequence of maps, Topologe and its Applications, 247 (2018)81-90.
\bibitem{wx5}P. Barrientos, A. Fakhari, Ergodicity of non-autonomous discrete systems with non-uniform expansion, Discrete and Continuous Dynamical Systems-Series B, 25(4) (2020), 1361-1382.
\bibitem{wx7}L. Liu, Y. Sun, Weakly mixing sets and transitive sets for non-autonomous discrete systems, Advances in Difference Equations, 217 (2014), 1-9.
\bibitem{q29}J. Pi, T. Lu, W. Anwar, Z. Mo, Further studies of topogogical transitiviey in nonautonomous discrete systems, Journal of Applied Analysis and Computation, 14(3)(2024)1508-1521.

\bibitem{wx9}H. Zeng, A note on multi-transitivity in non-autonomous discrete systems, arXiv:2505.24657.







\end{thebibliography}
\end{document}